\documentclass{notices}
\usepackage{amsfonts,amssymb,amsmath,amscd,amsthm}
\usepackage{graphicx} 
\usepackage{wrapfig}
\usepackage{enumitem}
\usepackage{subcaption}
\usepackage{url}
\usepackage{hyperref}
\usepackage{tikz-cd,tikz,tkz-graph}
\usetikzlibrary{decorations.pathreplacing,calc}
\usepackage[textwidth=2cm]{todonotes}
\usepackage{cleveref}
\newcommand{\DD}{\mathcal{D}}
\newcommand{\kk}{\Bbbk}
\newcommand{\NN}{\mathbb{N}}

\newcommand{\diam}{{\ensuremath{\mathrm{diam}}}}
\newcommand{\xg}[1]{x_{\{#1\}}}
\renewcommand\leq\leqslant
\renewcommand\geq\geqslant
\newtheorem{theorem}{Theorem}[section]
\newtheorem{lem}[theorem]{Lemma}

\newtheorem{conj}[theorem]{Conjecture}
\newtheorem{question}[theorem]{Question}
\theoremstyle{definition}

\newtheorem{exmp}[theorem]{Example}
\theoremstyle{remark}

\title{ Shuffling the Deck: Invariant Theory and the Graph Reconstruction Conjecture}
\author{
 Emilie Dufresne
 \affil{
 Lecturer in Algebra, Department of Mathematics, University of York, York YO10 5DD,UK. 
 }
 \and
 Gabriela Jeronimo
 \affil{
 Associate Professor, Department of Mathematics, University of Buenos Aires, Ciudad Universitaria, Pabellón I, 1428 Buenos Aires, Argentina.
 }
 \and
 Jenny Kenkel
 \affil{
 Assistant Professor, Grinnell College, 1115 8th Avenue, Grinnell, IA 50112, USA.
 }
 \and
 Haydee Lindo
 \affil{
 Associate Professor, Harvey Mudd College, 301 Platt Blvd., Claremont, CA 91711, USA. 
 }\and
 Nelly Villamizar
 \affil{
 Associate Professor, Department of Mathematics, Swansea University, Bay Campus, Fabian Way, Swansea SA1 8EN, UK.
 }
}

\begin{document}
 \maketitle

 \section*{The secret graph}
 Imagine there's a graph. We don't know what it is, but we have some clues. 
 If a particular vertex is removed from this graph, as well as all the edges connected to that vertex, then we're left with the graph in Figure \ref{fig:IntroCard1}. 
 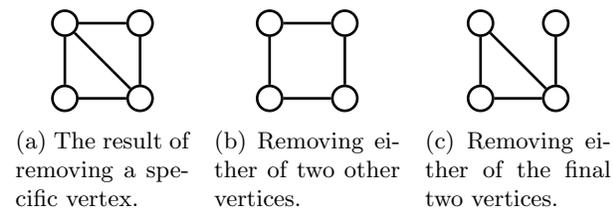
\begin{figure}[ht]
 \centering
 \begin{subfigure}[b]{0.14\textwidth}
 \centering
 \begin{tikzpicture}[line width=1pt]
 \tikzset{vertex/.style = {shape=circle,draw,minimum size=.3em}}
 \node[vertex] (a) at (0,1) {};
 \node[vertex] (b) at (0,0 ) {}; 
 \node[vertex] (c) at (1, 0) {}; 
 \node[vertex] (d) at (1,1) {}; 
 \draw (a)--(b);
 \draw (b)--(c);
 \draw (c)--(d);
 \draw (d)--(a);
 \draw (a)--(c); 
 \end{tikzpicture}
 \caption{The result of removing a specific vertex.}
 \label{fig:IntroCard1}
 \end{subfigure}
 \;
 \begin{subfigure}[b]{0.15\textwidth}
 \centering
 \begin{tikzpicture}[line width=1pt]
 \tikzset{vertex/.style = {shape=circle,draw,minimum size=.3em}}
 \node[vertex] (a) at (0,1) {};
 \node[vertex] (b) at (0,0 ) {}; 
 \node[vertex] (c) at (1, 0) {}; 
 \node[vertex] (d) at (1,1) {}; 
 \draw (a)--(b);
 \draw (b)--(c);
 \draw (c)--(d);
 \draw (d)--(a);
 \end{tikzpicture}
 \caption{Removing either of two other vertices.}
 \label{fig:IntroCard2}
 \end{subfigure}
 \;
 \begin{subfigure}[b]{0.15\textwidth}
 \centering
 \begin{tikzpicture}[line width=1pt]
 \tikzset{vertex/.style = {shape=circle,draw,minimum size=.3em}}
 \node[vertex] (a) at (0,1) {};
 \node[vertex] (b) at (0,0 ) {}; 
 \node[vertex] (c) at (1, 0) {}; 
 \node[vertex] (d) at (1,1) {}; 
 \draw (a)--(b);
 \draw (b)--(c);
 \draw (c)--(d);
 \draw (a)--(c);
 \end{tikzpicture}
 \caption{Removing either of the final two vertices.}
 \label{fig:IntroCard3}
 \end{subfigure}
 \caption{The deck of cards for the unknown graph.}
 \label{fig:deck}
 \end{figure}
 There are two different vertices in the graph whose removal would each yield the graph in Figure \ref{fig:IntroCard2}. 
 Finally, the last two vertices, when removed, each produce the graph in
 Figure \ref{fig:IntroCard3}. 
 Can you determine the secret graph? 
 
 You probably already figured out that the secret graph has five vertices. 
 With a bit more thought, you could determine it must have seven edges (more on this later). 
 At this point, you might take a brute-force approach, considering all ways to add a vertex and two edges to the graph in Figure \ref{fig:IntroCard1}. Eventually, you will find a graph that fits the description (see the end of this paper to have the problem spoiled). But can you be \emph{sure} what you've found is the only graph that works?

 The graph reconstruction conjecture asserts that you can be sure. 
 It posits that any two graphs with at least three vertices are isomorphic if and only if they have the same collection of isomorphism classes of vertex-deleted subgraphs.
 We call this multiset of unlabeled subgraphs the \emph{deck}, and each subgraph is a \emph{card}. 
 The conjecture claims that every graph is uniquely \emph{reconstructible} from its deck.
 
 For graphs on four vertices, it's easy to check by hand. 
 There are only $11$ non-isomorphic simple graphs, 
 and you can verify that all $11$ produce different decks before your coffee gets cold.
 With a thermos, you can probably check all $34$ connected graphs with five vertices. But to check all $156$ graphs with six vertices, you'll need several trips to the microwave.\footnote{Just imagine the taste of Brendan McKay's coffee after checking the conjecture for all $50$ trillion graphs with $13$ vertices \cite{McKay2022}.}
 This quickly gets out of hand, which is a major reason why the conjecture remains open today. 
 
 In this article, we will explore how \emph{invariant theory} provides an elegant framework for studying the graph reconstruction conjecture. The hope of this approach is to show that polynomials that distinguish between decks also distinguish between original graphs, thus translating a graph theoretic problem into an algebraic one.

 \section{A deck of cards}\label{sec:GRC}
 We now make these ideas precise in order to state the graph reconstruction conjecture.
 We have already defined the deck of a graph as the multiset of its (unlabeled) vertex-deleted subgraphs; each such subgraph is called a card. 

 Let $G$ be a finite simple (undirected) graph with vertex set $V(G)=\{v_1,\ldots, v_n\}$, indexed by $[n]=\{1,\ldots,n\}$, and edge set $E(G)\subseteq\{\{i,j\}\mid i,j\in [n],~i\neq j\}$. 
 
 A \emph{card} of $G$ is an unlabeled graph formed by removing one vertex $v_k$ of $G$, $k\in [n]$, and all the edges adjacent to it, and forgetting the labels of the remaining vertices. That is, a card is the isomorphism class of the vertex-deleted subgraph $G_k$ with $V(G_k)=V(G)\setminus \{v_k\}$ and $E(G_k)=E(G)\setminus \{\{i,j\}\mid i=k, \text{ or }j=k\}$.
 
 The \emph{deck} of $G$, written $\DD(G)$, is the multiset of all cards of $G$. There are $n$ cards in $ \DD(G)$, one for each vertex in $G$. Because the original labels of the vertices have been forgotten in each card, two cards formed from deleting different vertices may be in the same isomorphism class. This is why the deck is defined as a multiset, not a set, because we want to keep track of repeated cards. 
 
 \begin{exmp} Let $G$ be the graph with $n=4$ vertices in Figure \ref{fig:exampleGraph}.
 \begin{figure}[ht]
 \centering
 \begin{tikzpicture}[line width=1pt]
 \tikzset{vertex/.style = {shape=circle,draw,minimum size=.3em}}
 \node[vertex] (a) at (0,0) {$v_1$};
 \node[vertex] (b) at (1,0) {$v_2$};
 \node[vertex] (c) at (1,1) {$v_3$};
 \node[vertex] (d) at (0,1) {$v_4$};
 \draw[] (a)--(c) node[right]{};
 \draw[] (b)--(c) node[right]{};
 \draw[] (c)--(d) node[right]{};
 \draw[] (d)--(a) node[right]{};
 \end{tikzpicture}
 \caption{An example graph on four vertices.} 
 \label{fig:exampleGraph}
 \end{figure}
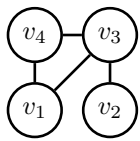
 The four subgraphs of $G$ produced by removing a vertex along with all its adjacent edges are in Figure 
 \ref{fig:vertexDeletedSubgraphs}.
 
 \begin{figure}[ht]
 \centering
 \begin{tikzpicture}[line width=1pt]
 \tikzset{vertex/.style = {shape=circle,draw,minimum size=.3em}}
 \node[vertex] (e) at (2,0) {$v_2$};
 \node[vertex] (f) at (2,1) {$v_3$};
 \node[vertex] (g) at (1,1) {$v_4$};
 \draw[] (e)--(f) node[right]{};
 \draw[] (g)--(f) node[right]{};
 \node[vertex] (h) at (4,0) {$v_1$};
 \node[vertex] (i) at (5,1) {$v_3$};
 \node[vertex] (j) at (4,1) {$v_4$};
 \draw[] (h)--(i) node[right]{}; 
 \draw[] (i)--(j) node[right]{};
 \draw[] (h)--(j) node[right]{};
 \begin{scope}[shift={(-4,-2)}] 
 \node[vertex] (n) at (5,0) {$v_1$};
 \node[vertex] (o) at (5,1) {$v_4$};
 \node[vertex] (p) at (6,0) {$v_2$};
 \draw[] (n)--(o) node[right]{};
 \node[vertex] (k) at (8,0) {$v_1$};
 \node[vertex] (l) at (9,0) {$v_2$};
 \node[vertex] (m) at (9,1) {$v_3$};
 \draw[] (k)--(m) node[right]{};
 \draw[] (l)--(m) node[right]{};
 \end{scope}
 \end{tikzpicture}
 \caption{The four vertex deleted subgraphs of the graph in Figure \ref{fig:exampleGraph}. }
 \label{fig:vertexDeletedSubgraphs}
 \end{figure}
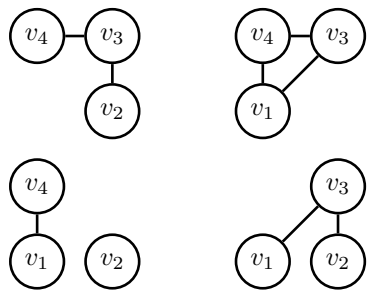

 \begin{figure}[ht]
 \centering
 \begin{tikzpicture}[line width=1pt]
 \tikzset{vertex/.style = {shape=circle,draw,minimum size=.3em}}
 \node[vertex] (e) at (2,0){};
 \node[vertex] (f) at (2,1) {};
 \node[vertex] (g) at (3,1) {};
 \node at (2.5,0){,};
 \draw[] (e)--(f) node[right]{};
 \draw[] (e)--(g) node[right]{};
 \draw[] (g)--(f) node[right]{};
 \node[vertex] (h) at (3.4,0) {};
 \node[vertex] (i) at (4.4,0) {};
 \node[vertex] (j) at (4.4,1) {};
 \node at (4.75,0){,};
 \draw[] (i)--(j) node[right]{};
 \draw[] (h)--(j) node[right]{};
 \node[vertex] (k) at (5.2,0) {};
 \node[vertex] (l) at (6.2,0) {};
 \node[vertex] (m) at (6.2,1) {};
 \node at (6.55,0){,};
 \draw[] (k)--(m) node[right]{};
 \draw[] (l)--(m) node[right]{};
 \node[vertex] (n) at (7,0) {};
 \node[vertex] (o) at (7.9,0) {};
 \node[vertex] (p) at (7,1) {};
 \draw[] (p)--(n) node[right]{};
 \draw [decorate,decoration={brace,amplitude=5pt,raise=4ex}]
 (2,-0.2) -- (2,1.2) node[midway,yshift=-3em]{};
 \draw [decorate,decoration={brace,amplitude=5pt,raise=4ex}]
 (2.2,-0.2) -- (2.2,1.2) node[midway,yshift=-3em]{};
 \draw [decorate,decoration={brace,amplitude=5pt,mirror, raise=4ex}]
 (7.6,-0.2) -- (7.6,1.2) node[midway,yshift=-3em]{};
 \draw [decorate,decoration={brace,amplitude=5pt, mirror, raise=4ex}]
 (7.8,-0.2) -- (7.8,1.2) node[midway,yshift=-3em]{}; 
 \end{tikzpicture}
 \vspace{-5mm}
 \caption{The deck of the graph in Figure \ref{fig:exampleGraph}.}\label{fig:example-deck}
 \end{figure}
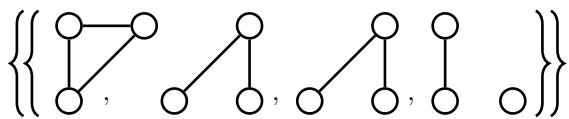
 Let $G_k$ denote the graph formed from $G$ by deleting vertex $v_k$. Observe from \ref{fig:vertexDeletedSubgraphs} that $G_1$ and $G_4$ isomorphic graphs so, the deck of $G$ is the multiset in Figure \ref{fig:example-deck}. \hfill$\diamond$
 \end{exmp}
 A graph $G$ is called \emph{reconstructible} if and only if $\DD(G)=\DD(G')$ implies that $G$ and $G'$ are isomorphic graphs. 
 To be reconstructible, a graph must have at least $2$ vertices; the two nonisomorphic graphs with two vertices, shown in Figure \ref{fig:twoVertices}, have the same deck, shown in Figure \ref{fig:deckTwoVertices}.
 
 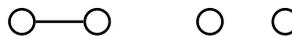
\begin{figure}[ht!]
 \begin{center}
 \begin{tikzpicture}[line width=1pt]
 \tikzset{vertex/.style = {shape=circle,draw, minimum size=.1em}}
 \node[vertex] (a) at (0,0) {};
 \node[vertex] (b) at (1,0) {};
 \node[vertex] (c) at (2.5,0) {};
 \node[vertex] (d) at (3.5,0) {};
 \draw[] (a)--(b) node[right]{};
 \end{tikzpicture}
 \caption{Nonisomorphic graphs with two vertices which have the same decks.}
 \label{fig:twoVertices}
 \end{center}
 \end{figure}  

 \begin{figure}[ht!]
 \centering
 \begin{tikzpicture}[line width=1pt]
 \tikzset{vertex/.style = {shape=circle,draw,minimum size=.3em}}
 \node[vertex] (e) at (0,0){};
 \node[vertex] (f) at (1,0) {};
  \node at (.5,0){,};

 
 \draw [decorate,decoration={brace,amplitude=5pt,raise=4ex}]
 (0,-0.4) -- (0,.5) node[midway,yshift=-3em]{};
 \draw [decorate,decoration={brace,amplitude=5pt,raise=4ex}]
 (-.2,-0.4) -- (-.2,.5) node[midway,yshift=-3em]{};
 \draw [decorate,decoration={brace,amplitude=5pt,mirror, raise=4ex}]
 (1.2,-0.4) -- (1.2,.5) node[midway,yshift=-3em]{};
 \draw [decorate,decoration={brace,amplitude=5pt, mirror, raise=4ex}]
 (1,-0.4) -- (1,.5) node[midway,yshift=-3em]{}; 
 \end{tikzpicture}
 \vspace{-5mm}
 \caption{The deck of both graphs with two vertices.}\label{fig:deckTwoVertices}
 \end{figure}
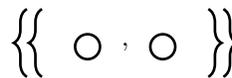
 The \emph{graph reconstruction conjecture} can be stated as follows.
 \begin{conj}[\cite{Kelly1957} and \cite{Ulam1960}]
 For $n\geq 3$, every simple undirected $n$-vertex graph is reconstructible.
 \end{conj}

 \section{What the deck reveals}
 Reading a graph's deck is like reading a hand of cards: some information is immediately visible, while other properties require more work. In this section we explore what can be determined from $\mathcal{D}(G)$
 alone.
 
 \subsection{Cards on the table}
 From the deck, the most apparent information about the original graph is the number of vertices: it equals the number of cards in $\DD(G)$, or equivalently, one more than the number of vertices in any single card. 
 The number of edges $|E(G)|$ can also be recovered. 
 Each edge of $G$ is absent from exactly two cards in the deck, namely those obtained by deleting each of its endpoints, and therefore appears in exactly $n-2$ cards. 
 It follows that 
 \[
 |E(G)|=\frac{\sum_{k=1}^n |E(G_k)|}{n-2}, 
 \]
 where the sum runs over all cards $G_k$ in the deck.

 This immediately settles the reconstruction conjecture for any class of graphs in which the number of edges determines the graph's isomorphism class: these include \emph{complete graphs} (every pair of vertices connected), \emph{almost complete graphs} (exactly one edge missing), empty graphs, and graphs with exactly one edge. A reference for reconstructability of these and other families of graphs is \cite{Manvel}.
 
 From the deck, we can also recover the degree of each vertex. Recall that the \emph{degree} $\deg(v_k)$ of a vertex $v_k\in V(G)$ is the number of edges that are incident to $v_k$.
 Since deleting $v_k$ removes precisely its incident edges, then we have 
 \[
 \deg(v_k) = |E(G)| - |E(G_k)|,
 \]
 where $G_k$ is the card obtained by deleting $v_k$. 
 
 We illustrate this with our earlier example (Figure \ref{fig:vertexDeletedSubgraphs}). 
 The deck has $4$ cards, so $G$ has $4$ vertices. 
 The cards have $3, 2, 2$ and $1$, edges respectively, giving 
 \[
 |E(G)|=\frac{3+2+2+1}{4-2} = 4.
 \]
 Subtracting each card's edge count $|E(G)|$ from the total number of edges then yields the degree of the corresponding deleted vertex, and so the degree multiset of $G$ is $ \{\!\{1,2,2,3\}\!\}$. 
 \begin{figure} 
 \centering
 \begin{tikzpicture}[line width=1.0pt]
 \tikzset{vertex/.style = {shape=circle,draw,minimum size=.3em}}
 \node[vertex] (a) at (0,1) {2};
 \node[vertex] (b) at (-.85,.3 ) {2}; 
 \node[vertex] (c) at (-.6, -.8) {1}; 
 \node[vertex] (d) at (.6,-.8) {1}; 
 \node[vertex] (e) at (.85,.3) {2}; 
 \draw[] (a)--(b) node[right]{};
 \draw[] (b)--(e) node[right]{};
 \draw[] (a)--(e) node[right]{};
 \draw[] (d)--(c) node[right]{};
 
 \begin{scope}[shift={(0.5,0)}] 
 
 \tikzset{vertex/.style = {shape=circle,draw,minimum size=.3em}}
 \node[vertex] (a2) at (2.5,1) {1};
 \node[vertex] (b2) at (1.65,.3 ) {2}; 
 \node[vertex] (c2) at (1.9, -.8) {2}; 
 \node[vertex] (d2) at (3.1,-.8) {2}; 
 \node[vertex] (e2) at (3.35,.3) {1}; 
 \draw[] (a2)--(b2) node[right]{};
 \draw[] (b2)--(c2) node[right]{};
 \draw[] (c2)--(d2) node[right]{};
 \draw[] (d2)--(e2) node[right]{};
 \end{scope}
 \end{tikzpicture}
 \caption{Each vertex is labeled with its degree. In this figure, two non-isomorphic graphs have the same degree multiset: $\{ \!\{ 1,1,2,2,2\} \!\}$.} \label{fig:sameDegSeq}
 \end{figure}
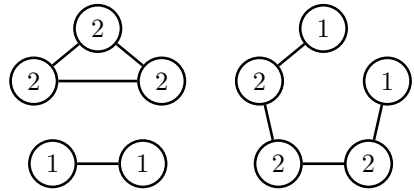 
 
 The degree multiset alone does not uniquely determine a graph. If it did, the graph reconstruction conjecture would have already been solved; see Figure \ref{fig:sameDegSeq} for two non-isomorphic graphs sharing the same degree multiset. Nonetheless, the degree multiset carries useful information. For instance, we can determine from the deck whether $G$ is \emph{regular} (that is, whether all vertices have the same degree), since this is visible directly from the degree multiset.
 
 A property of a graph is called \textit{recognizable} if it can be determined from the deck alone. As we have just seen, regularity is one such property. 

 \subsection{The two-step strategy}
 The general approach to proving the conjecture for a particular class $\mathcal{G}$ of graphs has two steps. 
 The first, \emph{recognition}, amounts to showing that the class is closed under deck equivalence: if given $G\in \mathcal{G}$ and $G'$ has the same deck as $G$, then $G'\in\mathcal{G}$ as well.
 The second, \emph{weak reconstruction}, amounts to showing that deck equivalence implies isomorphism within the class: if $G,G'\in \mathcal{G}$ have the same deck, then $G\cong G'$.

 We illustrate this strategy by proving the conjecture for Eulerian graphs. 
 Recall that a graph is \emph{Eulerian} (named after Euler's famous solution to the seven bridges problem) if it contains a cycle that visits every edge exactly once. 
 Such graphs admit a neat characterization: 
 $G$ is Eulerian if and only if it is connected and every vertex has even degree.

 Recognition for this class then follows immediately from two observations: connectivity is recognizable by the following lemma, and the parity of vertex degrees follows directly from the degree multiset.The following lemma is well-known to many that study graph reconstruction. 
 \begin{lem}\label{lemma}
A finite graph with at least 
3 vertices is connected if and only if at least two of its cards are connected.
 \end{lem} 
 \begin{proof}
 For the sake of contradiction, suppose that $G$ is a disconnected graph with at least 3 vertices, but that at least two of its cards are connected.

 Since a card is obtained by deleting a single vertex, $G$ can have at most two components. Let $G_i$ be one connected card, formed by deleting the vertex $v_i$. Since $G_i$ is connected but $G$ is not, the vertex $v_i$ must be isolated in $G$. Now consider a second connected card $G_j$, formed by deleting a different vertex $v_j \neq v_i$. 
 Since $v_i$ is isolated in $G$, it remains an isolated vertex in $G_j$, and since $G$ has at least three vertices, $G_j$ is not connected, a contradiction.

For the other direction, suppose that $G$ is connected. 
Now consider a longest path $P = v_1-v_2 - \cdots - v_k$ in $G$. 
We claim that the cards corresponding to removing $v_1$ and $v_k$ are connected. 
Suppose, for a contradiction, that removing $v_1$ would disconnect $G$. 
However, because $P$ is a longest path, all neighbors of $v_1$ must lie on $P$ (otherwise, we could extend $P$, contradicting maximality). 
But since all neighbors are on $P$, and $v_2$ remains, the card obtained by removing $v_1$ remains connected. The same argument applies to $v_k$. Thus the cards corresponding to $v_1$ and $v_k$ are connected, and hence $G$ has at least two connected cards.
\end{proof}

 We are now ready to complete the recognition step. We can determine whether a graph $G$ with at least three vertices is Eulerian by checking if at least two of its cards are connected and using the degree sequence to verify that every vertex has even degree. Hence the property of being Eulerian is recognizable from the deck alone.
 
 For weak reconstruction, to recover the original graph, we take any card, add a new vertex, and join it to every vertex in the card that has odd degree. Note that the same idea reconstructs regular graphs: having determined regularity from the degree multiset, we pick any card, add a new vertex, and join it to every vertex that does not yet have the correct degree.
 
 The second step, weak reconstruction, is not always as straightforward as these two examples suggest. In 2007, Bilinski, Kwon, and Yu proved that planar graphs are recognizable~\cite{Bilinskietal2007}, yet the full conjecture remains open for this class.
 
\section{What do we know?}
The graph reconstruction conjecture remains open, but significant progress has been made.
Over the years, the conjecture has also been proved for several classes of graphs, including trees, unicyclic graphs, cacti, outerplanar graphs, and maximal planar graphs. In this section we survey what is known.

\subsection{Beyond the thermos of coffee}\label{sec:beyondThermos}
Beyond specific families, probabilistic and computational methods have also shed light on the conjecture.

 A \emph{random graph} $G_p$ on $n$ vertices is one in which the edges are chosen independently with probability $p$, typically taken to be $1/2$. 
 We say that \emph{almost all} graphs have property $Q$ if the probability that $G_p$ has property $Q$ tends to $1$ as $n$ tends to infinity. 
 
 M\"uller~\cite{Muller1976} showed that almost all graphs are reconstructible, and Bollob\'as~\cite{bollobas--3reconstructibility} strengthened this by proving that almost every graph can be reconstructed from just three of its cards, a fact originally conjectured by Harary and Plantholt~ \cite{HararyPlanthold1985}. On the computational side, McKay~\cite{McKay2022} has verified the graph reconstruction conjecture for all graphs with up to 13 vertices. 
 
 Two useful reductions are also known. Yang~\cite{Yang_2-connectedGraphs} proved that the conjecture holds in general if and only if it holds for every 2-connected graph (that is, a graph $G$ such that every card of $G$ is connected). 
 Gupta et al.~\cite{Gupta} showed that it suffices to prove the conjecture for 
 graphs $G$ satisfying $\diam(G)=2$ or $\diam(G) = \diam(\bar{G})=2$, where $\diam(G)$ denotes the
 \emph{diameter} of $G$ (the maximum distance between any two vertices in the graph) and $\bar G$ is the complement of $G$.
 
 \subsection{Other card games}
 Several authors have focused on reconstructing partial information from the deck rather than the full graph.
 Others have studied variants of the graph reconstruction conjecture itself.
 Rather than deleting a single vertex at a time, one can form \emph{small cards} by deleting sets of $k \geq 1$
 vertices; a survey of reconstruction results in this setting can be found in~\cite{KostochkaWest2021}. 

Another variant considers decks with \emph{missing cards}: given a class of graphs that is reconstructible, or for which certain information is reconstructible, what is the largest number of cards that can be missing while reconstruction remains possible? See e.g.~~\cite{Groenlandetal2022} for recent results.
 
A closely related problem is \emph{edge-reconstruction}, where cards are formed by removing edges rather than vertices.
The edge-reconstruction conjecture claims that every graph with at least four edges is edge-reconstructible.
Greenwell~\cite{Greenwell} proved that reconstructibility implies edge-reconstructibility for graphs with no isolated vertices and at least four edges. 

Finally, recall that the total number of edges in a graph $G$ can be determined from its deck: \[
 |E(G)|=\frac{\sum_{k=1}^n |E(G_k)|}{n-2}. 
 \]
But for an arbitrary multiset of graphs with $n-1$ vertices, the total number of edges will not necessarily be a multiple of $n-2$. So not every multiset corresponds to a valid deck. The \emph{legitimate deck problem} posed by Harary~\cite{Harary63}, asks 
which multisets of $n$ graphs on $n-1$ vertices
can actually arise as the deck of some graph.  

 \section{Beyond simple graphs}\label{sec:WeightedG}
 A natural generalization is to consider \emph{directed graphs}, where every edge is given a direction. 
 The story here is short: Stockmeyer~\cite{Stockmeyer} proved that the directed graph reconstruction conjecture is false.
 \begin{exmp} 
 The two directed graphs in Figure~\ref{fig:directedGraphCounterExample} are not isomorphic:
 one contains a source vertex and the other a sink vertex, yet their directed decks are identical, as shown in Figure \ref{fig:deckDirectedGraph}. \hfill$\diamond$
 \begin{figure}[ht]
 \centering
 \begin{tikzpicture}[line width=1pt]
 \tikzset{vertex/.style = {shape=circle,draw,minimum size=.3em}};
 \tikzset{edge/.style = {->,>={Latex[length=1.8mm]}}};
 \node[vertex] (SW) at (-2, 0){};
 \node[vertex] (SE) at (-4, 0){};
 \node[vertex] (N) at (-3, 1.8){};
 \node[vertex] (C) at (-3, .7){};
 
 \draw[edge] (SE)--(SW);
 \draw[edge] (SW)--(N);
 \draw[edge] (N)--(SE);
 \draw[edge] (C)--(SE);
 \draw[edge] (C)--(SW);
 \draw[edge] (C)--(N);
 \node[vertex] (SW2) at (0, 0){};
 \node[vertex] (SE2) at (2, 0){};
 \node[vertex] (N2) at (1, 1.8){};
 \node[vertex] (C2) at (1, .7){};
 
 \draw[edge] (N2)--(SW2);
 \draw[edge] (SW2)--(SE2);
 \draw[edge] (SE2)--(N2);
 \draw[edge] (SE2)--(C2);
 \draw[edge] (SW2)--(C2);
 \draw[edge] (N2)--(C2);
 \end{tikzpicture} 
 \caption{Two non-isomorphic directed graphs with isomorphic decks.}
 \label{fig:directedGraphCounterExample}
 \end{figure}
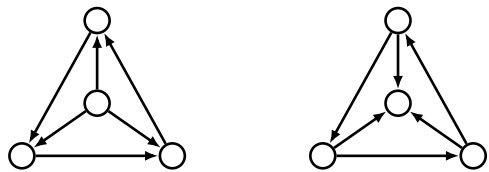

 \begin{figure}[ht]
 \centering
 \begin{tikzpicture}[line width=1pt]
 \tikzset{vertex/.style = {shape=circle,draw,minimum size=.3em}}
 \tikzset{edge/.style = {->,>={Latex[length=1.5mm]}}};
 \node[vertex] (e) at (2,0){};
 \node[vertex] (f) at (3,0) {};
 \node[vertex] (g) at (2.5,.9) {};
 \node at (3.25,0){,};
 
 \draw[edge] (e)--(f);
 \draw[edge] (f)--(g) node[right]{};
 \draw[edge] (g)--(e) node[right]{};

 \node[vertex] (h) at (3.7,0) {};
 \node[vertex] (i) at (4.7,0) {};
 \node[vertex] (j) at (4.2,.9) {};
 
 \draw[edge] (i)--(j) node[right]{};
 \draw[edge] (i)--(h) node[right]{};
 \draw[edge] (j)--(h) node[right]{};
 \node at (4.95,0){,};
 \node[vertex] (k) at (5.4,0) {};
 \node[vertex] (l) at (5.9,.9) {};
 \node[vertex] (m) at (6.4,0) {};
 \draw[edge] (k)--(l);
 \draw[edge] (k)--(m);
 \draw[edge] (m)--(l);
 \node at (6.65,0){,};
 \node[vertex] (n) at (7.1,0) {};
 \node[vertex] (o) at (8.1,0) {};
 \node[vertex] (p) at (7.6,.9) {};
 \draw[edge] (n)--(o) node[right]{};
 \draw[edge] (p)--(n) node[right]{};
 \draw[edge] (p)--(o) node[right]{};
 \draw [decorate,decoration={brace,amplitude=5pt,raise=4ex}]
 (2.1,-0.2) -- (2.1,1.2) node[midway,yshift=-3em]{};
 \draw [decorate,decoration={brace,amplitude=5pt,raise=4ex}]
 (2.3,-0.2) -- (2.3,1.2) node[midway,yshift=-3em]{};
 \draw [decorate,decoration={brace,amplitude=5pt,mirror, raise=4ex}]
 (7.8,-0.2) -- (7.8,1.2) node[midway,yshift=-3em]{};
 \draw [decorate,decoration={brace,amplitude=5pt, mirror, raise=4ex}]
 (8.0,-0.2) -- (8.0,1.2) node[midway,yshift=-3em]{}; 
 \end{tikzpicture}
 \caption{The directed deck of the graphs in Figure~\ref{fig:directedGraphCounterExample}.} \label{fig:deckDirectedGraph}
 \end{figure}
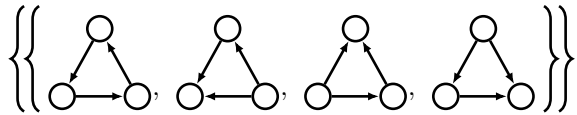 
 \end{exmp} 
\subsection{The weighted reconstruction conjecture}
\emph{Weighted graphs} are another reasonable generalization, in which each edge is assigned a weight, usually a non-zero number.
Weighted graphs arise naturally in combinatorial optimization, with the Travelling Salesman Problem being a prominent example. 
 
Figure \ref{fig:weightedGraph} shows a weighted graph with real weights: the edges $\{1,2\}$ and $\{2,4\}$ have weight 1, the edges $\{1,3\}$ and $\{3,4\}$ have weight -1, the edge $\{2,3\}$ has weight $3$, and the edge $\{1,4\}$ has weight 2. 

 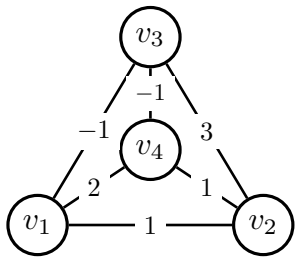
\begin{figure}[ht]
 \centering
 {\begin{tikzpicture} 
 \tikzset{vertex/.style = {shape=circle,draw,minimum size=.3em,very thick}}
 \tikzset{edge/.style = 
 {->,>={Latex[length=1.5mm]}}};
 \SetVertexNormal[LineWidth=1pt]
 \node[vertex] (A) at (0, 0){\scalebox{1.2}{$v_1$}};
 \node[vertex] (B) at (3, 0){\scalebox{1.2}{$v_2$}};
 \node[vertex] (C) at (1.5, 2.5){\scalebox{1.2}{$v_3$}};
 \node[vertex] (D) at (1.5, 1){\scalebox{1.2}{$v_4$}};
 \tikzstyle{LabelStyle}=[] 
 \tikzset{EdgeStyle/.style = {line width=1pt}}
 \tikzset{VertexStyle/.style = {line width=1pt}}
 \Edge[label=$1$](A)(B) 
 \Edge[label=$-1$](A)(C)
 \Edge[label=$2$](A)(D)
 \Edge[label=$3$](B)(C)
 \Edge[label=\scalebox{1}{$1$}](B)(D)
 \Edge[label=\scalebox{0.9}{$-1$}](C)(D)
 \end{tikzpicture}}
 \caption{An $\mathbb{R}$-weighted graph on four vertices} 
 \label{fig:weightedGraph}
 \end{figure}
 Following \cite{Pouzet77} and \cite{Thiery}, we then define a \textit{weighted card} to be the unlabeled weighted graph formed by removing one vertex and all adjacent edges, while preserving the weights on all remaining edges. 
 The \textit{weighted deck} is then the multiset of all weighted cards. 
 One can ask whether a weighted graph can be reconstructed from its weighted deck; this is the \textit{weighted} graph reconstruction conjecture. 
 
 The weighted graph reconstruction conjecture is stronger than the original: any simple graph can be viewed as a weighted graph with all the edges assigned the same weight, and any counterexample to the original conjecture would yield a counterexample to the weighted conjecture. 
 
 Conversely, any counterexample to the weighted reconstruction conjecture would involve only finitely many distinct edge weights. By relabeling these weights injectively, one could therefore obtain an equivalent counterexample whose weights lie in any sufficiently large subset of the field. Thus restricting attention to graphs whose weights belong to a large set (for example, the nonzero elements of the field) does not eliminate potential counterexamples.
 
 At first glance, passing to weighted graphs might seem to make an already hard problem harder. 
 The key insight, however, is that considering weighted graphs unlocks algebraic techniques that are inaccessible for simple graphs.
 \subsection{A vector space of graphs}
If $\kk$ is a field, the set $V_n$ of $\kk$-weighted graphs on $n$ vertices is a $\kk$-vector space with basis $\bigl\{e_{\{i,j\}}\mid i,j\in[n], i<j\bigr\}$, where each basis vector $e_{\{i,j\}}$ corresponds to a potential edge $\{i,j\}$. 
Any $\kk$-weighted graph is then represented as 
\[
\sum_{\{i,j\} \subseteq [n]} w_{\{i,j\}} e_{\{i,j\}},
\]
where $w_{\{i,j\}}\in\kk$ is the weight associated to the edge $\{i,j\}$, with $w_{\{i,j\}}=0$ indicating the absence of the edge. 
 For example, the vector corresponding to the graph in Figure \ref{fig:weightedGraph} is 
 \[ e_{\{1,2\}}-e_{\{1,3\}}+2e_{\{1,4\}}+e_{\{2,4\}}+3e_{\{2,3\}}-e_{\{3,4\}}.
 \]
 The symmetric group $S_n$ acts on $V_n$ by permuting vertex labels: $\sigma\cdot e_{\{i,j\}}=e_{\{\sigma(i), \sigma(j)\}}$, extended linearly. 
 Two $\kk$-weighted graphs $G$ and $H$, corresponding to vectors 
 $\sum_{\{i,j\}\subseteq [n]} w_{\{i,j\}} e_{\{i,j\}}$ and $\sum_{\{i,j\}\subseteq [n]} u_{\{i,j\}} e_{\{i,j\}}$, respectively, for $w_{\{i,j\}}, u_{\{i,j\}} \in \kk$, are isomorphic if and only if there exists $\sigma\in S_n$ such that 
 \[
 \sigma \cdot \biggl( \sum_{\{i,j\}\subseteq [n]} w_{\{i,j\}} e_{\{i,j\}}\biggr) = \sum_{\{i,j\}\subseteq [n]} u_{\{i,j\}} e_{\{i,j\}}.
 \] 
 \section{Telling orbits apart}\label{sec:InvThy}
 When a group acts on a vector space, \emph{invariants} are polynomial functions unchanged by the group action. They provide a way to study orbits, the distinct paths traced out by the action.
 For finite groups, something remarkable happens: we can always tell orbits apart using invariants and often do not need the entire ring of invariants to do it. 
A carefully chosen finite collection of invariants, called \emph{separating set}, suffices to distinguish every pair of orbits.

 In other words, if all invariants in the separating set agree on two vectors $u$ and $v$, then $u$ and $v$ must lie in the same orbit.
 This section explains this idea, building on work by Derksen-Kemper \cite{DerksenKemper} and Dufresne \cite{Dufresne2009}. 
 In Section~\ref{sec:4}, we connect it to graph reconstruction.
 
 To set the stage, suppose:
 \begin{itemize}[itemsep=0cm,topsep=0.1cm]
 \item $H$ is a finite group (such as the symmetric group $S_n$),
 \item $V$ is a finite-dimensional vector space over a field $\Bbbk$,
 \item For all $h\in H$, the function defined by $v \mapsto h \cdot v$ is an invertible linear transformation (in other words, the action of $H$ on $V$ defines a representation). 
 \end{itemize}
 
 Rather than working directly with vectors in $V$, we consider the ring of polynomial functions on $V$, denoted $\Bbbk[V]$. 
 This ring consists of all polynomial functions from $V$ to $\Bbbk$. 

 The group action of $H$ on $V$ induces an action on $\Bbbk[V]$, defined by
 \[
 (\sigma \cdot f)(v) = f(\sigma^{-1} \cdot v),
 \]
 for all $\sigma \in H, f \in \Bbbk[V]$, and $v \in V$. 
 The use of $\sigma^{-1}$ is needed to define a left action on $\Bbbk[V]$: 
it ensures that $(\sigma\tau) \cdot f = \sigma \cdot (\tau \cdot f)$ for 
all $\sigma, \tau \in H$, furthermore
 \begin{align*}
 (\sigma \cdot f )(\sigma \cdot v) = f(\sigma^{-1} \sigma \cdot v ) = f(v), 
 \end{align*}
 for all $v \in V$ and all $\sigma \in H$. 
 
 The \emph{ring of invariants}, denoted $\Bbbk[V]^H$, consists of all polynomials $f \in \Bbbk[V]$ that do not change when we apply the group action, that is, all polynomials such that $\sigma \cdot f = f$ for all $\sigma \in H$. 
 These invariants are constant on orbits in $V$:
 if $u$ and $v$ are elements of $V$ such that
 $u = \sigma \cdot v$ for some $\sigma \in H$, then 
 \[
 f(u) = f(\sigma\cdot v) = (\sigma^{-1}\cdot f)(v) = f(v),
 \]
 for all $f \in \Bbbk[V]^H$.
 
 The full ring of invariants distinguishes between orbits but can be unwieldy. A more practical notion is that of a separating set: rather than capturing all invariant information, we ask for a smaller collection of invariants that still suffices to tell orbits apart.
 
 \begin{question} Can we find a small subset of $\kk[V]^H$ that still separates orbits?
 \end{question}
 
 To make this precise, suppose there exists a polynomial $f \in \Bbbk[V]^H$ such that $f(u) \ne f(v)$ for some $u,v \in V$. 
 Then $u$ and $v$ must lie in different orbits, and we say that $f$ \emph{separates} $u$ and $v$. 
 
 A subset $E \subseteq \Bbbk[V]^H$ is called a \emph{separating set} if
 for all $u, v \in V$, if there exists $f \in \Bbbk[V]^H \text{ with } f(u) \ne f(v)$, then there exists $g \in E$ with $g(u) \ne g(v)$. In other words, $E$ separates every pair of vectors that an invariant can separate.

 This notion, introduced by Derksen and Kemper in the first edition of the book \cite{DerksenKemper}, is useful because separating sets are often much smaller and easier to work with than generating sets for the full ring $\Bbbk[V]^H$.
 
 For finite groups, a beautiful fact is that orbits can always be separated~\cite{DerksenKemper}*{Section 2.3.1}, 
 and in many cases separating sets can be described explicitly.
 
 \begin{exmp}
 [A cyclic group action]\label{eg:SepSet}
 Consider the field $\Bbbk = \mathbb{C}$ and the action of the cyclic group $C_4 = \langle \sigma \rangle$ of order 4 on $V = \mathbb{C}^2$ defined by
 $\sigma \cdot (v_1, v_2) = (i v_1, i v_2)$, where $i^2 = -1$.
 
 Let $x_1, x_2$ be the coordinate functions on $V$, that is, the linear transformations such that $x_{j}(e_k)=1$ if $j=k$ and $0$ otherwise. 
 The induced action on $\mathbb{C}[x_1, x_2]$ is:
 \[
 \sigma \cdot x_1 = -i x_1, \quad \sigma \cdot x_2 = -i x_2.
 \]
 One can verify this by checking that $(\sigma \cdot x_1)(v) = x_1(\sigma^{-1} \cdot v)$.
 Now consider the set:
 \[
 E = \bigl\{x_1^4, x_1^3 x_2, x_2^4\bigr\}.
 \]
 Each polynomial in $E$ is invariant under the action of $C_4$, and together they form a separating set. To see why, suppose $u = (u_1, u_2)$ and $v = (v_1, v_2)$ satisfy $f(u) = f(v)$ for all $f \in E$. 
 A straightforward computation shows that $u$ must be a scalar multiple of $v$ by a fourth root of unity; hence $u$ and $v$ lie in the same orbit under the action of $C_4$.
 
 This illustrates that separating sets can be strictly smaller than generating sets; the set $E$ does \emph{not} generate the full invariant ring. For example, the polynomials $x_1^2 x_2^2$ and $x_1 x_2^3$ are also invariants, but they are not in the subalgebra generated by $E$. 
In general, multiple distinct separating sets may exist. For example, the set \(\{x_1^4, x_1x_2^3,x_2^4\}\) is also a separating set. \hfill$\diamond$
 \end{exmp}

{Why does this matter?}
This example hints at why separating sets are the right tool for studying isomorphism classes of graphs. 
When graphs are represented as points in a vector space, 
the symmetric group $S_n$ acts by permuting vertex labels. Invariants under this action correspond precisely to graph properties that do not depend on how vertices are labeled. 

A separating set then allows us to:
\begin{itemize}[itemsep=0cm,topsep=0.1cm]
 \item test whether two graphs are isomorphic, that is, whether they lie in the same orbit;
 \item investigate whether a graph can be reconstructed from partial information, such as its deck.
\end{itemize}
Although these problems are computationally hard in general, separating invariants provide a powerful theoretical tool: even without a full description of the invariant ring, we know a small, manageable set of functions suffices to determine orbit membership.

\section{Orbit sums}\label{sec:4}
We may now describe how invariant theory provides an approach to the Graph Reconstruction Conjecture. 
As we saw in Section \ref{sec:WeightedG}, weighted graphs on $n$ vertices can be viewed as vectors in the space $V_n$, and the symmetric group acts on this space by permuting vertex labels. Two graphs are isomorphic precisely when they lie in the same orbit of this action.

Invariant theory studies polynomial functions on $V_n$
that are constant on these orbits. Such functions therefore describe properties of graphs that do not depend on the labeling of the vertices.

If, as before, $e_{\{i,j\}}$ denote the standard basis vectors of $V_n$, indexed by pairs $1\leq i<j\leq n$, we write $x_{\{i,j\}}$ for the dual coordinate functions. 
The polynomial ring 
\[
\kk[V_n]=\kk\bigl[x_{i,j}\mid 1\leq i<j\leq n\bigr],
\] 
consists of polynomials in one variable $x_{i,j}$ for each possible edge $\{i,j\}$.  

A monomial 
\[
x^{M}=\prod x_{\{i,j\}}^{m_{\{i,j\}}}
\] 
can be interpreted as encoding a weighted graph $M$. 
The power $m_{i,j}$ records the weight of the edge $\{i,j\}$; when $m_{i,j}=0$, the edge $\{i,j\}$ is absent from the graph.
 In this way, monomials correspond naturally to $\NN$-weighted graphs on $n$ vertices.

The group $S_n$ acts on the polynomial ring $\kk[V_n]$ by permuting indices: 
\[
\sigma\cdot x_{\{i,j\}}=x_{\{\sigma(i),\sigma(j)\}},
\] 
giving a representation of $S_n$ on $\kk[V_n]$. 
Thus the orbit of a monomial corresponds exactly to the set of all weighted graphs obtained from $M$ by relabeling the vertices.

A particularly natural family of invariants arises by summing a monomial over its orbit under the action of $S_n$. 
For a monomial $x^M$, the \emph{orbit sum} $O_{S_n}\bigl(x^{M}\bigr)$ is defined as the sum of all distinct monomials in the orbit of $x^M$ on $S_n$.

By construction, each orbit sum is invariant under the action of $S_n$. 
Orbit sums therefore provide a way to construct invariant polynomials.

\begin{exmp}\label{ex:4.1}
 Consider the case $n=3$. The variables in $k[V]$ are 
 \[
 x_{\{1,2\}}, \ \ x_{\{1,3\}},\ \ \text{and } x_{\{2,3\}}.
 \]
Let 
\[
x^M=x_{\{1,2\}}x_{\{2,3\}},
\]
be the monomial that corresponds to the path $M$ from $v_1$ to $v_3$ through $v_2$; see figure \ref{fig:pathGraph}.  
The orbit of this monomial under the action of $S_3$ consists of the three monomials 
\[x_{\{1,2\}}x_{\{2,3\}}, \\ x_{\{1,2\}}x_{\{1,3\}},\text{ and } x_{\{1,3\}}x_{\{2,3\}}.\]
The orbit sum is therefore
\[
O_{S_3}(x^M)=x_{\{1,2\}}x_{\{2,3\}} + x_{\{1,2\}}x_{\{1,3\}} + x_{\{1,3\}}x_{\{2,3\}}.
\]
Each term corresponds to a labeling of the same underlying graph. \hfill$\diamond$
\end{exmp}

\begin{figure}[ht]
 \centering
 \begin{tikzpicture}
 \tikzset{vertex/.style = {shape=circle,draw,minimum size=.3em,very thick}};
 \SetVertexNormal[LineWidth=1pt]
 \node[vertex] (A) at (0,0){$v_1$};
 \node[vertex] (B) at (1.5,0){$v_2$};
 \node[vertex] (C) at (3,0){$v_3$};
 
 \draw (A)--(B);
 \draw (B)--(C);
 
 \end{tikzpicture}
 \caption{The graph $M$ corresponding to the monomial $x^M=x_{\{1,2\}}x_{\{2,3\}}$ in Example \ref{ex:4.1}.}
 \label{fig:pathGraph}
\end{figure}
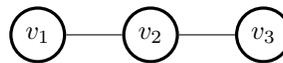
Orbit sums therefore provide natural invariants of graphs, since they depend only on the isomorphism class of the underlying graph.
The polynomials $O_{S_n}(x^{M })$ are invariant by construction, and in fact they form a basis for $\kk[V_n]^{S_n}$ as a $\kk$-vector space. 
Indeed, every invariant polynomial $f\in \kk[V_n]^{S_n}$ can be written as a finite $\kk$-linear combination of monomials. 
Since $f$ is invariant, all monomials in the same $S_n$-orbit must appear with the same coefficient. 
Consequently, $f$ can be expressed as a $\kk$-linear combination of orbit sums $O_{S_n}(x^M)$.

In particular, orbit sums generate the invariant ring $\kk[V_n]^{S_n}$ as a $\kk$-algebra. 
If $\kk$ has characteristic zero (or characteristic larger than $n$), then Noether's bound implies that the ring of invariants is generated in degree at most $|S_n|=n!$.
Thus, although the invariant ring is finitely generated, the resulting generating set is extremely large. 

Despite these limitations, orbit sums offer a conceptually simple family of invariants with a clear combinatorial interpretation. As we now explain, evaluating these invariants on a graph leads to counting subgraphs.

\begin{exmp}\label{ex:4.2}
For $n=3$, the generating set for $\kk[V_n]^{S_n}$ given by orbit sums of all multigraphs with at most $3!=6$ edges consists of $22$ polynomials in $\kk\bigl[\xg{1,2},\xg{1,3},\xg{2,3}\bigr]$. In this case, however, the invariant ring can be described explicitly.
 
Indeed, if we set
\[
X_1=\xg{2,3}, \qquad X_2=\xg{1,3}, \qquad X_3=\xg{1,2},
\]
then the action of $S_3$ is given by
\[
\sigma\cdot X_i = X_{\sigma(i)}
\]
for every $\sigma\in S_3$. Thus $\kk[V_n]^{S_n}$ is the ring of symmetric polynomials in $X_1,X_2,X_3$, and so it is generated by the elementary symmetric polynomials:
\begin{itemize}[itemsep=0cm,topsep=0.1cm]
\item $X_1+X_2+X_3 = O_{S_3}\bigl(\xg{2,3}\bigr)$,
 \item $X_1X_2+X_1X_3+X_2X_3 = O_{S_3}\bigl(\xg{2,3}\xg{1,3}\bigr)$,
 \item $X_1X_2X_3 = O_{S_3}\bigl(\xg{2,3}\xg{1,3}\xg{1,2}\bigr)$.
 \end{itemize}
 The first generator is the orbit sum of a monomial corresponding to a graph with one edge, the second to a graph with two edges, and the third to a graph with three edges. These graphs are illustrated in \Cref{fig:3Vertices}.
 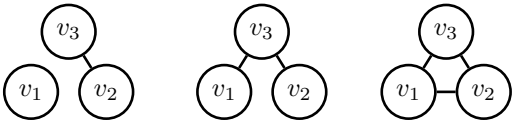
\begin{figure}[ht]
 \centering
 \begin{tikzpicture}[line width=1pt]
 \tikzset{vertex/.style = {shape=circle,draw,minimum size=.3em}}
 \node[vertex] (a) at (0,0) {$v_1$};
 \node[vertex] (b) at (1,0) {$v_2$};
 \node[vertex] (c) at (.5,.8) {$v_3$};
 \draw[] (b)--(c) node[right]{};
 \end{tikzpicture} \qquad
 \begin{tikzpicture}[line width=1pt]
 \tikzset{vertex/.style = {shape=circle,draw,minimum size=.3em}}
 \node[vertex] (a) at (0,0) {$v_1$};
 \node[vertex] (b) at (1,0) {$v_2$};
 \node[vertex] (c) at (.5,.8) {$v_3$};
 \draw[] (b)--(c) node[right]{};
 \draw[] (a)--(c) node[right]{};
 \end{tikzpicture}\qquad
 \begin{tikzpicture}[line width=1pt]
 \tikzset{vertex/.style = {shape=circle,draw,minimum size=.3em}}
 \node[vertex] (a) at (0,0) {$v_1$};
 \node[vertex] (b) at (1,0) {$v_2$};
 \node[vertex] (c) at (.5,.8) {$v_3$};
 \draw[] (b)--(c) node[right]{};
 \draw[] (a)--(c) node[right]{};
 \draw[] (a)--(b) node[right]{};
 \end{tikzpicture}
 \caption{Graphs on three vertices corresponding to the orbit sums of monomials with one, two, and three edges.}\label{fig:3Vertices}
 \end{figure}
 
 Thus, although the general construction gives $22$ orbit sums, in this case the invariant ring is generated by only $3$ of them.\hfill$\diamond$
\end{exmp}

Example \ref{ex:4.2} shows that, for $n=3$, the invariant ring can be generated by orbit sums of simple graphs. 
Pouzet asked whether this remains true in general: do orbit sums of simple graphs generate the invariant ring under an $S_n$-action \cite{Pouzet77}? 
Thi\'ery answered this question in the negative: for $n\ge 5$, there exists a homogeneous invariant of degree $4$ that does not lie in the $\kk$-algebra generated by the orbit sums of simple graphs \cite[Thm.~11.2.5]{ThieryThesis}. 
The proof is non-constructive and uses the graded structure of the invariant ring.

\section{Orbit sums count subgraphs}\label{sec:countSubgraphs}

A crucial observation is that, when $\kk$ has characteristic zero, orbit sums of simple graphs count subgraphs. More precisely, orbits sums serve as indicator functions for the presence of  classes of subgraphs. 

We identify a simple graph $G$ with a point of $V_n$ by setting
\[
x_{\{i,j\}}(G)=
\begin{cases}
 1 & \text{if }\{i,j\}\text{ is an edge of }G,\\
 0 & \text{otherwise}.
\end{cases}
\]
If $G$, $G'$ are two simple graphs on $n$ vertices, and 
$x^{G'}$ is the monomial corresponding to $G'$, then 
$x^{G'}(G)=1$ exactly when $G$
contains $G'$ as a subgraph (with the given labeling). 
Consequently, evaluating the orbit sum $O_{S_n}(x^{G'})$ on $G$ counts the number of subgraphs of $G$ that are isomorphic to $G'$.

\begin{exmp}
 Consider the complete graph $G$ on three vertices and the graph $G'$ consisting of two edges forming a path through $v_2$ as in Figure \ref{fig:pathGraph}.
 
 The monomial corresponding to $G'$ is
$
 x^{G'}=x_{\{1,2\}}x_{\{2,3\}} .
$
Evaluating this monomial at the vector representing $G$ gives $1$, since both edges appear in $G'$.
 
The orbit sum $O_{S_3}\bigl(x^{G'}\bigr)$ adds the contributions coming from all relabelings of this path. Consequently, evaluating $O_{S_3}\bigl(x^{G'}\bigr)$ on $G$ counts the number of subgraphs of $G$ that are isomorphic to $G'$. In this case, it counts the number of paths contained in the triangle.\hfill$\diamond$
\end{exmp}

\begin{exmp}
Consider $G$, the graph corresponding to $e_{\{1,2\}}+e_{\{2,3\}}$ and $G'$, the graph corresponding to $e_{\{1,3\}}$. 
Under the original labeling, $G'$ is not a subgraph of $G$. 
The corresponding monomial is 
\[
x^{G'}=x_{\{1,3\}},\] and therefore \[x^{G'}\bigl(e_{\{1,2\}}+e_{\{2,3\}}\bigr)=0.
\]
However, 
\[
O_{S_n}\bigl(x^{G'}\bigr)= x_{\{1,3\}}+x_{\{1,2\}}+x_{\{ 2,3\}}.
\]
Evaluating this polynomial on $G$ gives
 \begin{equation*} 
 O_{S_n}\bigl(x^{G'}\bigr)(G) = 0 + 1 + 1 =2.
 \end{equation*}
Thus $G$ contains two subgraphs that are isomorphic to $G'$, corresponding to the two edges of the path. The graphs $G$ and $G'$
are shown in Figure~\ref{combinedFigure}.\hfill$\diamond$
 
 \begin{figure}[ht]
 \centering
 \begin{subfigure}{0.26\textwidth}
 \centering
 \begin{tikzpicture}[line width=1pt]
 \tikzset{vertex/.style = {shape=circle,draw,minimum size=.3em}}
 
 \node[vertex] (a) at (0,0) {$v_1$};
 \node[vertex] (b) at (1,0) {$v_2$};
 \node[vertex] (c) at (.5,.8) {$v_3$};
 
 \draw[] (b)--(c) node[right]{};
 \draw[] (a)--(b) node[right]{};
 \end{tikzpicture}
 \caption{Graph \scalebox{0.95}{$G = e_{\{1,2\}}+e_{\{2,3\}}$}. } 
 \label{G1}
 \end{subfigure}
 \;
 \begin{subfigure}{0.2\textwidth}
 \centering
 \begin{tikzpicture}[line width=1pt]
 \tikzset{vertex/.style = {shape=circle,draw,minimum size=.3em}}
 
 \node[vertex] (a) at (0,0) {$v_1$};
 \node[vertex] (b) at (1,0) {$v_2$};
 \node[vertex] (c) at (.5,.8) {$v_3$}; 
 \draw[] (a)--(c) node[right]{};
 \end{tikzpicture}
 \caption{Graph $G' = e_{\{1,3\}}$.} 
 \label{identifyingSubgraphs}
 \end{subfigure}
 \caption{Note that $G'$ is not a subgraph of $G$ when the labels are preserved, but $G$ has two subgraphs that are isomorphic to $G'$ when the labels are removed.}
 \label{combinedFigure}
 \end{figure}
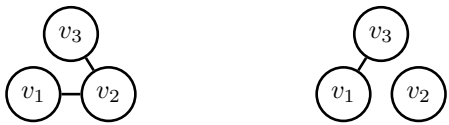
\end{exmp}

These observations also suggest an algorithmic viewpoint. 
A well-known problem in graph theory is the \textit{subgraph isomorphism problem}: given graphs $G$ and $G'$, what's the quickest procedure to determine whether $G$ has a subgraph that is isomorphic to $G'$? 
We have just demonstrated a technique using invariant theory: to determine whether $G'$ appears as a subgraph of $G$, one can construct the orbit sum $O_{S_n}(x^{G'})$ and evaluate this polynomial on $G$. 

For small $n$, this approach looks quite reasonable. 
However, the number of variables and number of terms in the polynomial grows extremely fast with $n$. 
To see this, consider the case where $G'$ is a cycle of length $n$. 
If $G$ is any simple graph on $n$ vertices, then evaluating $O_{S_n}(x^{G'})(G)$ counts the number of Hamiltonian cycles (that is, cycles that visit every vertex exactly once) in $G$.
In particular, it determines whether $G$ contains a Hamiltonian cycle. 
Since this decision problem is well known to be an NP-complete problem, it is not surprising that the polynomial involved is very large: it is a homogeneous polynomial of degree $n$ in $\binom{n}{2}$ variables and contains 
 \[\frac{(n-1)!}{2}\]
terms. 

What does this have to do with the reconstruction conjecture? The key idea is to consider orbit sums of monomials corresponding to multigraphs that contain at least one isolated vertex. 
We denote such a multigraph on $n$ vertices by $M*$, where $M$ is any multigraph on $n-1$ vertices and the star indicates the isolated vertex. 

One can show that an orbit sum $O_{S_n}(x^{M*})$ takes the same value on any two $\kk$-weighted graphs with the same deck (see \cite[Prop.~14.3.2]{ThieryThesis}). 
In fact, the following stronger result holds.
\begin{theorem}[{\cite[Thm.~14.3.7]{ThieryThesis}}]
 Two $\kk$-weighted graphs $G$ and $G'$ have the same deck if and only if 
 \[
 O_{S_n}\bigl(x^{M*}\bigr)(G)=O_{S_n}\bigl(x^{M*}\bigr)\bigl(G'\bigr)
 \] 
 for every multigraph $M*$ on n vertices with at least one isolated vertex.
\end{theorem}

These results suggest an invariant theory approach to the graph reconstruction conjecture. Pouzet conjectured that the orbit sums $O_{S_n}\bigl(x^{M*}\bigr)$, where $M*$ ranges over all multigraphs with at least one isolated vertex, generate the invariant ring $\kk[V_n]^{S_n}$ \cite{Pouzet77}. This conjecture, were it true, would imply the graph reconstruction conjecture for $\kk$-weighted graphs. 

Unfortunately, this is not the case in general: the orbit sums corresponding to multigraphs with at least one isolated vertex do not generate $\kk[V_n]^{S_n}$ for all $n$ (see \cite{ThieryThesis}). 
Using a computational, though still non-constructive, method, Thi\'ery established this failure for $ 11\leq n\leq 18$. 
Note, however, that this does not provide a counterexample to the graph reconstruction conjecture for $\kk$-weighted graphs. 
To prove reconstruction, it is not necessary for these orbit sums to generate the full invariant ring; it would suffice for them to form a separating set. 
As Example~\ref{eg:SepSet} illustrates, separating sets need not generate the entire ring of invariants. 

\begin{exmp}
 For $n=3$, the situation is especially simple. In this case the action of $S_3$ identifies $\kk[V_3]^{S_3}$ with the ring of symmetric polynomials in three variables. In characteristic zero, the power sums generate this ring via Newton's identities. Since these power sums arise from orbit sums of multigraphs with an isolated vertex, Pouzet's conjecture holds for $n=3$, and hence so does the graph reconstruction conjecture for $\kk$-weighted graphs on three vertices.\hfill$\diamond$
\end{exmp}

As mentioned in Section~\ref{sec:beyondThermos}, Bollob\'as \cite{bollobas--3reconstructibility} proved that almost all graphs can be reconstructed, up to graph isomorphism, from just three of its cards. 
Here ``almost all" means that, as the number of vertices tends to infinity, the probability that a random graph has this property tends to 1. 
There is an analog for $\kk$-weighted graphs, where ``almost all" is understood in the algebraic-geometric sense: outside a proper closed subset, or equivalently, on a dense open subset. 

\begin{theorem}[cf. \cite{ThieryThesis}]
 There is a dense open subset $U\subseteq V_n$ such that every $\kk$-weighted graph in $U$ can be reconstructed, up to isomorphism, from any 
 three of its cards.
\end{theorem}

\begin{proof}
 Let $U$ be the set of all $\kk$-weighted graphs whose edge weights are pairwise distinct.
 This is the principal open set defined by
 \[
 \prod_{\{i,j\}\neq \{k,\ell\}}\bigl(x_{\{i,j\}}-x_{\{k,\ell\}}\bigr).
 \] 
 Now let $G\in U$, and choose any three cards from its deck. 
 Without loss of generality, suppose these are the cards $C_1,C_2,C_3$, obtained by deleting vertices 1, 2, and 3, respectively. 
 Because all edge weights are distinct, the vertices can be identified uniquely by comparing the weights of adjacent edges. 
 Using $C_1$ and $C_2$ one can reconstruct all of $G$ except for the weight of the edge $\{1,2\}$. 
 The remaining weight is determined from the third card $C_3$. 
\end{proof}
Note that the dense open set $U$ in this proof contains no simple graphs, since in a simple graph all present edges have the same weight. 
Thus, this result does not recover Bollob\'as' theorem. 

\section{Enough cards \\to determine the truth}
The Graph Reconstruction Conjecture has captured the attention of many mathematicians. 
One reason is the accessibility of the problem. 
Another is the intuitiveness of the question: like many problems across mathematics, it asks whether complete information can be recovered from partial information. 
A third reason is how convincingly true it appears. Numerous classes of graphs, dense families, and even trillions of individual graphs have been shown to satisfy the conjecture, yet a general proof remains elusive.

Recasting the reconstruction conjecture in the language of invariant theory allows us to study the problem through an algebraic lens. Instead of asking directly whether a graph can be reconstructed from its deck, we ask whether certain invariant polynomials determined by the deck suffice to distinguish graphs up to isomorphism.

Although the computations involved can be formidable, this algebraic viewpoint sheds light on which graph properties can be recovered from a deck and highlights the role of symmetry underlying the reconstruction problem. Whether this viewpoint will ultimately lead to a proof remains unknown.

It may be that the invariant-theoretic approach can provide only partial insight. But any proof of a conjecture that has remained open for more than half a century will likely rely on many such partial insights woven together. Each new perspective adds another card to the deck. Eventually, we may have enough cards to determine the truth.

The graph from our opening puzzle is shown in Figure~\ref{fig:theSecretGraph}.

\begin{figure}[ht]
 \centering
 \begin{tikzpicture}[line width=1pt]
 \tikzset{vertex/.style = {shape=circle,draw,minimum size=1em}}
 \tikzset{edge/.style = {->,>={Latex[length=1.5mm]}}};
 \SetVertexNormal[LineWidth=1pt]
 \node[vertex] (a) at (0,1) {};
 \node[vertex] (b) at (-.9,.3 ) {}; 
 \node[vertex] (c) at (-.6, -.8) {}; 
 \node[vertex] (d) at (.6,-.8) {}; 
 \node[vertex] (e) at (.9,.3) {}; 
 \draw[] (a)--(b) node[right]{};
 \draw[] (b)--(c) node[right]{};
 \draw[] (c)--(d) node[right]{};
 \draw[] (d)--(e) node[right]{};
 \draw[] (e)--(a); 
 \draw[] (a)--(c) node[]{};
 \draw[] (b)--(e);
 \end{tikzpicture}
 \caption{The secret graph from the introduction.}
 \label{fig:theSecretGraph} 
\end{figure}
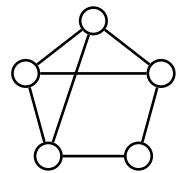 

\section*{Acknowledgments}

The authors thank the organizers and sponsors of the Women in Commutative Algebra III workshop (WiCA III), held in Oaxaca in 2024, where this collaboration began. Continued support was provided by SLMath -- the Simons Laufer Mathematical Sciences Institute, which hosted us for two weeks as part of its Summer Research in Mathematics program in 2025, and by the International Centre for Mathematical Sciences (ICMS) in Edinburgh, which hosted our Research in Groups program for two weeks in January 2026. Gabriela Jeronimo also gratefully acknowledges partial funding from the London Mathematical Society (LMS). Jenny Kenkel would like to thank the Henry Luce Foundation for its generous funding. 



\begin{bibdiv}
	\begin{biblist}
		
		\bib{Bilinskietal2007}{article}{
			author={Bilinski, M.},
			author={Kwon, Y.},
			author={Yu, X.},
			title={On the reconstruction of planar graphs},
			date={2007},
			ISSN={0095-8956,1096-0902},
			journal={J. Combin. Theory Ser. B},
			volume={97},
			number={5},
			pages={745\ndash 756},
			url={https://doi.org/10.1016/j.jctb.2006.12.005},
			review={\MR{2344137}},
		}
		
		\bib{bollobas--3reconstructibility}{article}{
			author={Bollob\'{a}s, B.},
			title={Almost every graph has reconstruction number three},
			date={1990},
			ISSN={0364-9024,1097-0118},
			journal={J. Graph Theory},
			volume={14},
			number={1},
			pages={1\ndash 4},
			url={https://doi.org/10.1002/jgt.3190140102},
			review={\MR{1037416}},
		}
		
		\bib{DerksenKemper}{book}{
			author={Derksen, H.},
			author={Kemper, G.},
			title={Computational invariant theory},
			series={Encyclopaedia of Mathematical Sciences},
			publisher={Springer, Heidelberg},
			date={2015},
			volume={130},
			ISBN={978-3-662-48420-3; 978-3-662-48422-7},
			url={https://doi.org/10.1007/978-3-662-48422-7},
			review={\MR{3445218}},
		}
		
		\bib{Dufresne2009}{article}{
			author={Dufresne, E.},
			title={Separating invariants and finite reflection groups},
			date={2009},
			ISSN={0001-8708},
			journal={Adv. Math.},
			volume={221},
			number={6},
			pages={1979\ndash 1989},
			url={https://doi.org/10.1016/j.aim.2009.03.013},
			review={\MR{2522833}},
		}
		
		\bib{Groenlandetal2022}{article}{
			author={Groenland, C.},
			author={Johnston, T.},
			author={Kupavskii, A.},
			author={Meeks, K.},
			author={Scott, A.},
			author={Tan, J.},
			title={Reconstructing the degree sequence of a sparse graph from a
				partial deck},
			date={2022},
			ISSN={0095-8956,1096-0902},
			journal={J. Combin. Theory Ser. B},
			volume={157},
			pages={283\ndash 293},
			url={https://doi.org/10.1016/j.jctb.2022.07.004},
			review={\MR{4463034}},
		}
		
		\bib{Gupta}{article}{
			author={Gupta, S.},
			author={Mangal, P.},
			author={Paliwal, V.},
			title={Some work towards the proof of the reconstruction conjecture},
			date={2003},
			ISSN={0012-365X,1872-681X},
			journal={Discrete Math.},
			volume={272},
			number={2-3},
			pages={291\ndash 296},
			url={https://doi.org/10.1016/S0012-365X(03)00198-5},
			review={\MR{2009550}},
		}
		
		\bib{Greenwell}{article}{
			author={Greenwell, D.},
			title={Reconstructing graphs},
			date={1971},
			ISSN={0002-9939,1088-6826},
			journal={Proc. Amer. Math. Soc.},
			volume={30},
			pages={431\ndash 433},
			url={https://doi.org/10.2307/2037710},
			review={\MR{286699}},
		}
		
		\bib{Harary63}{incollection}{
			author={Harary, F.},
			title={On the reconstruction of a graph from a collection of subgraphs},
			date={1964},
			booktitle={Theory of {G}raphs and its {A}pplications ({P}roc. {S}ympos.
				{S}molenice, 1963)},
			publisher={Publ. House Czech. Acad. Sci., Prague},
			pages={47\ndash 52},
			review={\MR{175111}},
		}
		
		\bib{HararyPlanthold1985}{article}{
			author={Harary, F.},
			author={Plantholt, M.},
			title={The graph reconstruction number},
			date={1985},
			ISSN={0364-9024,1097-0118},
			journal={J. Graph Theory},
			volume={9},
			number={4},
			pages={451\ndash 454},
			url={https://doi.org/10.1002/jgt.3190090403},
			review={\MR{890233}},
		}
		
		\bib{Kelly1957}{article}{
			author={Kelly, P.},
			title={A congruence theorem for trees},
			date={1957},
			ISSN={0030-8730,1945-5844},
			journal={Pacific J. Math.},
			volume={7},
			pages={961\ndash 968},
			url={http://projecteuclid.org/euclid.pjm/1103043674},
			review={\MR{87949}},
		}
		
		\bib{KostochkaWest2021}{article}{
			author={Kostochka, A.},
			author={West, D.},
			title={On reconstruction of graphs from the multiset of subgraphs
				obtained by deleting {$\ell$} vertices},
			date={2021},
			ISSN={0018-9448,1557-9654},
			journal={IEEE Trans. Inform. Theory},
			volume={67},
			number={6},
			pages={3278\ndash 3286},
			url={https://doi.org/10.1109/TIT.2020.2983678},
			review={\MR{4289319}},
		}
		
		\bib{Muller1976}{article}{
			author={Müller, V.},
			title={Probabilistic reconstruction from subgraphs},
			date={1976},
			ISSN={0010-2628,1213-7243},
			journal={Comment. Math. Univ. Carolinae},
			volume={17},
			number={4},
			pages={709\ndash 719},
			review={\MR{441789}},
		}
		
		\bib{Manvel}{article}{
			author={Manvel, B.},
			title={On reconstructing graphs from their sets of subgraphs},
			date={1976},
			ISSN={0095-8956},
			journal={J. Combinatorial Theory Ser. B},
			volume={21},
			number={2},
			pages={156\ndash 165},
			url={https://doi.org/10.1016/0095-8956(76)90056-3},
			review={\MR{424619}},
		}
		
		\bib{McKay2022}{article}{
			author={McKay, B.},
			title={Reconstruction of small graphs and digraphs},
			date={2022},
			ISSN={1034-4942,2202-3518},
			journal={Australas. J. Combin.},
			volume={83},
			pages={448\ndash 457},
			review={\MR{4446370}},
		}
		
		\bib{Pouzet77}{article}{
			author={Pouzet, M.},
			title={Quelques remarques sur les r\'esultats de {Tutte} concernant le
				probl\`eme de {Ulam}},
			language={fr},
			date={1977},
			journal={Publications du D\'epartement de math\'ematiques (Lyon)},
			volume={14},
			number={2},
			pages={1\ndash 8},
			url={https://www.numdam.org/item/PDML_1977__14_2_1_0/},
			review={\MR{522658}},
		}
		
		\bib{Stockmeyer}{article}{
			author={Stockmeyer, P.},
			title={The falsity of the reconstruction conjecture for tournaments},
			date={1977},
			ISSN={0364-9024,1097-0118},
			journal={J. Graph Theory},
			volume={1},
			number={1},
			pages={19\ndash 25},
			url={https://doi.org/10.1002/jgt.3190010108},
			review={\MR{453584}},
		}
		
		\bib{Thiery}{article}{
			author={Thi{\'e}ry, N.},
			title={Algebraic invariants of graphs; a study based on computer
				exploration},
			date={2000},
			journal={SIGSAM Bull.},
			volume={34},
			pages={9\ndash 20},
			note={\url{https://hal.science/hal-00348388v1}},
		}
		
		\bib{ThieryThesis}{thesis}{
			author={Thi{\'e}ry, N.},
			title={Invariants algébriques de graphes et reconstruction. une étude
				expérimentale.},
			type={PhD thesis},
			address={Lyon, France},
			date={1999},
			note={\textnumero~d'ordre 167-99},
		}
		
		\bib{Ulam1960}{book}{
			author={Ulam, S.},
			title={A collection of mathematical problems},
			series={Interscience Tracts in Pure and Applied Mathematics},
			publisher={Interscience Publishers, New York-London},
			date={1960},
			volume={no. 8},
			review={\MR{120127}},
		}
		
		\bib{Yang_2-connectedGraphs}{article}{
			author={Yang, Y.},
			title={The reconstruction conjecture is true if all {$2$}-connected
				graphs are reconstructible},
			date={1988},
			ISSN={0364-9024,1097-0118},
			journal={J. Graph Theory},
			volume={12},
			number={2},
			pages={237\ndash 243},
			url={https://doi.org/10.1002/jgt.3190120214},
			review={\MR{940833}},
		}
		
	\end{biblist}
\end{bibdiv}
\end{document}